\documentclass[11pt]{amsproc}
\usepackage{hyperref,xy,doi}
\usepackage[dvipsnames]{xcolor}
\usepackage{tikz}
\usepackage{relsize}
\xyoption{all}
\numberwithin{equation}{section}
\hypersetup{
	colorlinks   = true, 
	urlcolor     = blue, 
	linkcolor    = magenta, 
	citecolor   = blue 
}

\newtheorem{innercustomthm}{Theorem}
\newenvironment{customthm}[1]
{\renewcommand\theinnercustomthm{#1}\innercustomthm}
{\endinnercustomthm}

\makeatletter
\@namedef{subjclassname@2020}{%
	\textup{2020} Mathematics Subject Classification}
\makeatother

\newtheorem{theorem}{Theorem}[section]
\newtheorem{corollary}[theorem]{Corollary}
\newtheorem{lemma}[theorem]{Lemma}
\newtheorem{prop}[theorem]{Proposition}
\theoremstyle{definition}
\newtheorem{definition}[theorem]{Definition}
\theoremstyle{remark}
\newtheorem{remark}[theorem]{Remark}
\newtheorem{example}[theorem]{Example}

\DeclareMathOperator{\Sym}{Sym}

\newcommand{\CC}{\mathbb C}
\newcommand{\sgn}{\mathrm{sgn}}

\newcommand{\mom}[1]{\left\langle #1\right\rangle}
\newcommand{\smo}[1]{\left\{#1\right\}}

\newcommand{\cc}[3]{\mathcal C_#1(#2,#3)}
\newcommand{\cd}[3]{\mathcal B_#1(#2,#3)}
\newcommand{\rr}[3]{\textcolor{red}{\mathcal R_#1(#2,#3)}}
\newcommand{\gr}[3]{\textcolor{RoyalBlue}{\mathcal B_#1(#2,#3)}}
\newcommand{\rd}[3]{\textcolor{red}{\mathcal U_#1(#2,#3)}}
\newcommand{\gd}[3]{\textcolor{RoyalBlue}{\mathcal L_#1(#2,#3)}}
\newcommand{\mch}[2]{\ensuremath{\left(\kern-.3em\left(\genfrac{}{}{0pt}{}{#1}{#2}\right)\kern-.3em\right)}}
\newcommand{\bigchar}{\mathlarger{\mathlarger{\chi}}}
\newcommand{\pring}{\CC[X_1,X_2,\dotsb]}

\newcommand{\expn}{\text{exp}}
\title[Restriction Coefficients]{Some Restriction Coefficients for the\\Trivial and Sign Representations}

\author[Narayanan]{Sridhar P. Narayanan}
\address{Indian Institute of Technology Bombay, Mumbai}
\email{sridharp.narayanan@gmail.com}

\author[Paul]{Digjoy Paul}
\address{Chennai Mathematical Institute, Chennai}
\email{digjoypaul@gmail.com}

\author[Prasad]{Amritanshu Prasad}
\address{The Institute of Mathematical Sciences (HBNI), Chennai}
\email{amri@imsc.res.in}

\author[Srivastava]{Shraddha Srivastava}
\address{Indian Institute of Science, Bangalore}
\email{maths.shraddha@gmail.com}

\begin{document}

\begin{abstract}
  We use character polynomials to obtain a positive combinatorial interpretation of the multiplicity of the sign representation in irreducible polynomial representations of $GL_n(\CC)$ indexed by two-column and hook partitions.
  Our method also yields a positive combinatorial interpretation for the multiplicity of the trivial representation of $S_n$ in an irreducible polynomial representation indexed by a hook partition.
\end{abstract}
\maketitle
\section{Introduction}
\label{sec:intro}
The representation theory of symmetric groups and general linear groups lies at the heart of algebraic combinatorics.
The irreducible polynomial representations of the general linear group $GL_n(\CC)$ {are given by Weyl modules}  $W_\lambda(\CC^n)$, for partitions $\lambda$ with at most $n$ parts, while the irreducible representations of the symmetric group $S_n$ {are given by Specht modules} $V_\mu$, for partitions $\mu$ of $n$.

Decomposing into irreducibles the restriction of $W_\lambda(\CC^n)$ to the subgroup of permutation matrices, which is isomorphic to $S_n$, we have
\begin{displaymath}
\text{Res}^{GL_n(\CC)}_{S_n} W_\lambda(\CC^n)= \bigoplus_{\mu} V_{\mu}^{\oplus r_{\lambda \mu}},
\end{displaymath}
where the multiplicities $r_{\lambda \mu}$ are called the \emph{restriction coefficients}. A combinatorial description of these coefficients is a natural and, as yet, unresolved question, which we will henceforth call \emph{the restriction problem}. 
 
Littlewood \cite{MR95209} proved that
\begin{align}\label{al:identity}
s_{\mu}[1+h_1+h_2+\dotsb]= \sum _{\lambda}r_{\lambda \mu}s_\lambda,
\end{align}
where $s_{\mu}[1+h_1+h_2+\dotsb]$ is the plethystic substitution of the sum of the complete homogeneous symmetric functions $h_i$ into the Schur function $s_\mu(x_1,x_2,\dotsc)$.
Concretely, it is the substitution of the monomials occurring in $1+h_1+h_2+\dotsb$ into the variables $x_1,x_2,\dotsc$ of $s_\mu$.
The reader is referred to \cite{MR2765321,MR3443860,MR1676282} for a complete treatment of symmetric functions and plethysm.
In \cite{Narayanan2021}, we interpreted Littlewood's identity in terms of an induction functor from representations of $S_n$ to polynomial representations of $GL_n(\CC)$.

In recent years, many new approaches to the restriction problem have been developed. Assaf--Speyer  \cite{assaf}  and Orellana--Zabrocki \cite{ORELLANA2021107943} independently defined the unique basis for the ring of symmetric functions such that the restriction coefficients appear in the linear expansion of Schur functions with respect to this basis. In \cite{ALCO_2022__5_5_1165_0}, Orellana--Zabrocki--Saliola--Schilling related the restriction problem to the subalgebra of uniform block permutations within the partition algebra. In \cite{ALCO_2021__4_2_189_0},  Heaton--Sriwongsa--Willenbring proved that every irreducible representation of $S_n$ occurs in at least one of a certain specified family of Weyl modules parameterized by two-row partitions.

Character polynomials encode the character values of certain families of $S_n$ representations across $n$ (see Section~\ref{sec:prelims}).
The application of character polynomials to the restriction problem was suggested by Garsia and Goupil in \cite{MR2576382}.
In \cite{ALCO_2021__4_4_703_0}, we computed character polynomials for the family $(W_\lambda(\CC^n))$ as $S_n$-representations across all $n$.
Using these character polynomials, we interpreted restriction coefficients for the trivial and sign representation as signed sums of vector partitions.
This allowed us to derive necessary and sufficient conditions for the occurrence of the trivial representation of $S_n$ in $W_{\lambda}(\CC^n)$ when $\lambda$ is a partition with two rows, or two columns, or is hook-shaped. 

In this paper, we introduce the notion of signed moment of a character polynomial at each nonnegative integer $n$ (see Definition \ref{def:mom}). In Theorems~\ref{th:gen2col},~\ref{th:gen_hook},~\ref{th:mom_gen_hook} we find the generating functions from which we deduce {the following} main results of this paper.
An expository account of some of these ideas can be found in \cite{frobchar}.
 
 \begin{customthm}{A}
 \label{th:sign_2row}
  For $n\geq 2$, let $(k,l)'$ be the conjugate of the partition $(k,l)$, where  $0 < l \leq k \leq n$. Then the sign representation of $S_n$ occurs in $W_{(k,l)'}(\CC^n)$ if and only if $(k,l)\in \{(n-1,0),(n,0),(n-1,1),(n,1)\}$. In all these cases it occurs with multiplicity one.
\end{customthm}
 
 \begin{customthm}{B}
  \label{theorem:hook-sign}
  For all $a,b\geq 0$ with $b+1\leq n$, the multiplicity of the sign representation of $S_n$ in $W_{(a+1,1^b)}(\CC^n)$ is the number of pairs $(\lambda,\mu)$ of partitions such that
  \begin{enumerate}
  \item $\lambda=(\lambda_1\geq \cdots \geq \lambda_{b}\geq 0)$ and $\mu=(\mu_1>\cdots >\mu_{n-b}\geq 0)$,
  \item $\lambda_1+\cdots+\lambda_b+\mu_1+\cdots +\mu_{n-b}=a+1$, and $\mu_1>\lambda_1$.
  \end{enumerate}
  \end{customthm}

 \begin{customthm}{C}
 \label{theorem:hook-triv}
For all $a,b\geq 0$ with $b+1\leq n$, the multiplicity of trivial representation of $S_n$ in $W_{(a+1,1^b)}(\CC^n)$ is the number of pairs $(\lambda,\mu)$ of partitions such that
  \begin{enumerate}
  \item $\lambda=(\lambda_1\geq \cdots \geq \lambda_{n-b}\geq 0)$ and $\mu=(\mu_1>\cdots >\mu_{b}\geq 0)$,
  \item $\lambda_1+\cdots+\lambda_{n-b}+\mu_1+\cdots +\mu_{b}=a+1$ and $\mu_1<\lambda_1-1$.
  \end{enumerate}
\end{customthm}
Owing to Littlewood's formula \eqref{al:identity}, each of these results reveals a term in the expansion of $s_\nu[1+h_1+h_2+\dotsb]$ in the basis of Schur functions, where $\nu$ is either of the partitions $(1^n)$ or $(n)$ (Corollaries~\ref{cor:pleth_twocol}, \ref{cor:pleth_hook_sign}, and \ref{cor:pleth_hook_triv}).
\section{Character polynomials and their moments}
\label{sec:prelims}

For $i\geq 1$, define the class function $X_i$  on $S_n$ as
\begin{displaymath}
X_i(\sigma)= \text{number of cycles of length } i \text{ in }\sigma.
\end{displaymath}
A polynomial in the ring $\pring$ is called a \emph{character polynomial}.
For each $n\geq 1$, let $V_n$ be a representation of $S_n$. 
The sequence $(V_n)$ is  \emph{eventually polynomial} if there exists a polynomial $p \in \CC[X_1,X_2,\dotsc]$ and an integer $N\geq 0$ such that, for each $n\geq N$ and each $\sigma \in S_n$
\begin{displaymath}
p(X_1(\sigma), X_2(\sigma),\dotsc)= \bigchar_{V_n}(\sigma),
\end{displaymath}
where $\bigchar_{V_n}$ is the character of the representation $V_n$. 

The polynomial $p$ is unique since if $f(X_1(\sigma),X_2(\sigma),\dotsc)=0$ for a polynomial $f$ and for all permutations $\sigma$, then it is zero for all sequences of nonnegative integer values, and must thus be identically zero. The polynomial $p$ is the character polynomial associated to the sequence $(V_n)$.
\begin{example}
\label{eg:charpoly}
Let $V_n= \Sym^2(\CC^n)$.
$\Sym^2(\CC^n)$ has  a basis indexed by multisets (collections of possibly repeated elements) of size $2$ drawn from $[n]=\{1,\dotsc,n\}$.
We calculate the character polynomial of $(V_n)$ by counting the number of fixed points of the action of an arbitrary permutation $\sigma \in S_n$ on this basis. 

A permutation $\sigma$ acts on a multiset $\{i,j\}$ as
$$\sigma \cdot \{i,j\}= \{\sigma(i),\sigma(j)\}.$$

A multiset is fixed by $\sigma$ if $\{i,j\}= \{\sigma(i),\sigma(j)\}$. Thus either $i=\sigma(i)$ and $j=\sigma(j)$, or $i=\sigma(j)$ and $j=\sigma(i)$. There are $X_1(\sigma)+{X_1(\sigma) \choose 2}$ ways of picking $i,j$ from the $X_1(\sigma)$ fixed points of $\sigma$ for the former case and $X_2(\sigma)$ ways of picking $i,j$ to be elements of a 2-cycle of $\sigma$ for the latter. 
Thus
\begin{displaymath}
p= X_1+{X_1 \choose 2}+X_2.
\end{displaymath}
\end{example}
\begin{definition}
\label{def:mom}
Given a polynomial $p\in \CC[X_1,X_2,\dotsc]$, the \emph{moment} of $p$ is defined to be
\begin{displaymath}
\mom{p}_n= \frac{1}{n!}\sum_{\sigma \in S_n} p(X_1(\sigma),X_2(\sigma),\ldots).
\end{displaymath}
 
The \emph{signed moment} of $p$ is defined to be
\begin{displaymath}
\smo{p}_n= \frac{1}{n!}\sum_{\sigma \in S_n} \text{sgn}(\sigma)p(X_1(\sigma),X_2(\sigma),\ldots),
\end{displaymath}
where $\text{sgn}(\sigma)$ is the value of the sign character on $\sigma$. 
\end{definition}
\begin{remark}
  Let $p$ be the character polynomial of the eventually polynomial sequence of representations $(V_n)$. Clearly $\mom{p}_n$ and $\smo{p}_n$ are respectively the multiplicities of the trivial and sign representation of $S_n$ in $V_n$.
  Moments were introduced and their generating functions were studied in \cite{ALCO_2021__4_4_703_0}.
  Here we extend that theory to signed moments.
\end{remark}

Given a partition $\lambda=(\lambda_1\geq \lambda_2\geq \dotsb\geq \lambda_l> 0)$, the \emph{length} of the partition $l(\lambda)=l$ is the number of nonzero parts of the partition. Let $a_i$ be the number of parts of $\lambda$ of size $i$. Then $\lambda$ is expressed in the \emph{exponential notation} as
\begin{displaymath}
\lambda= 1^{a_1}2^{a_2}\dotsb.
\end{displaymath}
Define $$z_\lambda= \frac{|\lambda|!}{\prod_i i^{a_i}a_i!},$$ the cardinality of the centralizer of a permutation of cycle-type $\lambda$. Here $|\lambda|=\sum_i ia_i$ is the size of the partition $\lambda$. Let $\text{Par}$ denote the set of all partitions and let $\text{Par}_n$ denote the subset of partitions of $n$ for any nonnegative integer $n$. 
\begin{definition}
\label{def:choice_prod}
For a partition $\alpha=1^{a_1}2^{a_2}\dotsb$, define the following elements in $\pring$: $${X \choose \alpha} = \prod_i {X_i \choose a_i}$$ and $$\mch{X}{\alpha}= \prod_i \mch{X_i}{a_i},$$ where $\mch{X_i}{a_i}:= {X_i+a_i-1 \choose a_i}$.
\end{definition}
The sets $\{{X \choose \alpha}\mid \alpha \in \text{Par}\}$ and $\{{\mch{X}{\alpha}}\mid \alpha \in \text{Par}\}$ are bases of $\mathbb{C}[X_1,X_2,\dotsc]$. The former is called the \emph{binomial basis} of $\mathbb{C}[X_1,X_2,\dotsc]$.

\begin{theorem}
\label{th:binom}
  For a partition $\alpha$ we have
  \begin{displaymath}
    \smo{\binom X\alpha}_n =
    \begin{cases}
      \frac{\sgn(\alpha)}{z_\alpha} &\text{if } n \in \{|\alpha|,|\alpha|+1\},\\
      0 &\text{otherwise},
    \end{cases}
  \end{displaymath}
where $\sgn(\alpha)$ is the value of the sign representation on a permutation of type $\alpha$. 
\end{theorem}
\begin{proof}
The signed moment $\{ {X \choose \alpha} \}= \frac{1}{n!}\sum_{w \in S_n}\text{sgn}(w){\beta(w) \choose \alpha}$, where $\beta(w)$ is the cycle-type of the permutation $w$. Gathering together permutations in each conjugacy class we have
\begin{align*}
  \sum_{n\geq 0} \smo{\binom X\alpha}_n v^n & = \sum_{n\geq 0} \sum_{\beta\vdash n=1^{b_1}2^{b_2}\dotsc} \frac 1{z_\beta} (-1)^{b_2+b_4+\dotsb}\binom\beta\alpha v^n\\
                                          & = \sum_{b_i\geq a_i} \prod_i \frac {(-1)^{(i-1)b_i}}{i^{b_i}b_i!}\frac{b_1!}{a_i!(b_i-a_i)!} v^{ib_i}\\
                                          & = \prod_i \frac{(-1)^{(i-1)a_i}v^{ia_i}}{i^{a_i}a_i!}\sum_{c_i\geq 0} \frac{(-1)^{(i-1)c_i}v^{ic_i}}{i^{c_i}c_i!}\\
                                          & = \frac{\sgn(\alpha)v^{|\alpha|}}{z_\alpha}\sum_{n\geq 0}\frac {v^n}{n!} \sum_{w\in S_n} \sgn(w)\\
  & = \frac{\sgn(\alpha)v^{|\alpha|}}{z_\alpha}(1+v),
\end{align*}
where the final equality is due to the signed sum of permutations being nonzero precisely when $n=0,1$.

The coefficient of $v^n$ is thus $\frac{\sgn(\alpha)}{z_\alpha}$ if $n=|\alpha|$ or $n=|\alpha|+1$, and is $0$ otherwise.
\end{proof}

Theorem \ref{th:binom} allows us thus to make general conclusions about the presence of the sign representation.

\begin{corollary}
 Let $(V_n)$ be an eventually polynomial sequence of representations of $S_n$. Then the sign representation does not occur in $V_n$ for sufficiently large $n$. In particular, the sign representation does not occur in the restriction of $W_\lambda(\CC^n)$ to $S_n$ for any $n>|\lambda|+1$. 
\end{corollary}

\begin{proof}
The first statement is a straightforward consequence of Theorem \ref{th:binom}. From \cite[Section~2.3]{ALCO_2021__4_4_703_0}, we know that the the expansion of the character polynomial of $W_\lambda(\CC^n)$ in the binomial basis involves partitions at most as large as $|\lambda|$.
\end{proof}

\section{Combinatorial interpretations for some restriction coefficients}
\label{sec:combint}
This section contains the main results of this paper. Here we give combinatorial interpretations for the multiplicity of the sign representation of $S_n$ in $W_\lambda(\CC^n)$ when $\lambda$ is a two-column or hook partition. The proof of the latter result is easily modified to yield the multiplicity of the trivial representation of $S_n$ when $\lambda$ is a hook partition. We recall the following character polynomials from \cite{ALCO_2021__4_4_703_0}

\begin{theorem}
\label{th:charHE}
 Let $\Sym^{k}(\CC^n)$ and $\bigwedge^{l}(\CC^n)$ be the $k$-th symmetric and $l$-th exterior powers of $\CC^n$ respectively. Then
\begin{enumerate}
\item The sequence of representations $(\Sym^{k}(\CC^n))$ is eventually polynomial, and the character polynomial is
\begin{displaymath}
H_k= \sum_{\alpha \vdash k} \mch{X}{\alpha}.
\end{displaymath}
\item The sequence of representations $(\bigwedge^{l}(\CC^n))$ is eventually polynomial, and the character polynomial is
\begin{displaymath}
E_l= \sum_{\alpha \vdash k} (-1)^{|\alpha|-l(\alpha)}{X\choose \alpha}.
\end{displaymath}
\end{enumerate} 
\end{theorem}

\subsection{Multiplicity of sign representation when $\lambda$ has two columns}
\label{subsec:2col}

Let $(k,l)'$ be the conjugate of the partition $(k,l)$, where  $0 < l \leq k \leq n$. The character polynomial associated to $(W_{(k,l)'}(\CC^n))$ is defined using the dual Jacobi-Trudi identity (see \cite[Corollary 7.16.2]{MR1676282})
\begin{align*}
S_{(k,l)'}= E_kE_l- E_{k+1}E_{l-1}.
\end{align*}

The signed moment of $S_{(k,l)'}$ is thus the difference of signed moments, for which we have the following generating functions.
\begin{theorem}
\label{th:gen2col}
For integers $k,l\geq 0$, the signed moment of $E_kE_l$ is
\begin{displaymath}
\sum_{k,l,n}\{E_{k}E_{l}\}_nu^kv^lz^n=\frac{(1+z)(1+uvz)}{(1-uz)(1-vz)}. 
\end{displaymath}
\end{theorem}
\begin{proof}
By Theorem~\ref{th:charHE}, the signed moment $\{E_kE_l\}_n$ is
\begin{displaymath}
\{E_kE_l\}_n= \sum_{\alpha \vdash n}\frac{1}{z_\alpha} (-1)^{|\alpha|-l(\alpha)}\sum_{\beta \vdash k, \gamma \vdash l} (-1)^{|\beta|-l(\beta)}(-1)^{|\gamma|-l(\gamma)}{\alpha \choose \beta}{\alpha \choose \gamma}.
\end{displaymath}
Which is the coefficient of $[u^kv^lz^n]$ in
\begin{align*}
\prod_{i\geq 1}  \sum_{a_i,b_i,c_i \geq 0}\frac{1}{i^{a_i}a_i!}(-1)^{(i-1)a_i+(i-1)b_i+(i-1)c_i}{a_i \choose b_i}{a_i \choose c_i}u^{ib_i}v^{ic_i}z^{ia_i}.
\end{align*}
The inner sum is $\expn(\frac{-(-z)^i(1-(-u)^i)(1-(-v)^i)}{i})$. Thus
\begin{align*}
\{E_kE_l\}_n&=[u^kv^lz^n]\prod_{i \geq 1}\expn(\frac{-(-z)^i(1-(-u)^i)(1-(-v)^i)}{i})\\
&=[u^kv^lz^n]\expn(\sum_{i \geq 1}\frac{(uz)^i+(vz)^i}{i})\expn(\sum_{i \geq 1}\frac{-(-z)^i-(-uvz)^i}{i})\\
&=[u^kv^lz^n]\frac{(1+z)(1+uvz)}{(1-uz)(1-vz)},
\end{align*}
as required.
\end{proof}
The following corollary of Theorem \ref{th:gen2col} expresses the signed moment of $S_{(k,l)'}$ as the difference of the cardinality of two sets. 
\begin{corollary}
\label{cor:qstar}
  Let $q^*_n(k,l)$ denote the number of partitions of $(k,l)$ into $n$ parts, where each part is in $\{(0,0), (1,0), (0,1),(1,1)\}$, and where $(0,0)$ and $(1,1)$ occur at most once, while $(1,0)$ and $(0,1)$ may occur multiple times.
  Then the signed moment of $S_{(k,l)'}$ for $k\geq l$ is given by
  \begin{displaymath}
    \smo{S_{(k,l)'}}_n = \begin{cases}
      q^*_n(k,l)-q^*_n(k+1,l-1) & \text{if } l>0,\\
      q^*_n(k,l) & \text{otherwise}.
    \end{cases}
  \end{displaymath}
\end{corollary}
\begin{lemma}
\label{lem:qval}
  For integers $k\geq l$ we have
  \begin{displaymath}
    q^*_n(k,l) =
    \begin{cases}
      1 & \text{if } l>0\text{ and } k+l\in \{n-1,n+1\},\\
      2 & \text{if } l>0\text{ and } k+l=n,\\
      1 & \text{if } l=0\text{ and } k\in \{n-1,n\},\\
      0 & \text{otherwise}.
    \end{cases}
	  \end{displaymath}
\end{lemma}
\begin{proof}
  If $l>0$, and $k+l=n-1$, then the only partition is
  \begin{displaymath}
    (k,l) = (0,0) +k (1,0) + l (0,1) 
  \end{displaymath}
  If $l>0$ and $k+l=n$, then there are two partitions
  \begin{align*}
    (k,l) & = (0,0) + (1,1) + (k-1)(1,0) + (l-1)(0,1)\\
          & = k(1,0) + l(0,1).
  \end{align*}
  If $l>0$ and $k+l=n+1$, then the only partition is
  \begin{displaymath}
    (k,l) = (1,1) + (k-1)(1,0) + (l-1)(0,1).
  \end{displaymath}
  If $l=0$, and $k=n-1$ then the only partition is
  \begin{displaymath}
    (k,0) = (0,0)+k(1,0).
  \end{displaymath}
  If $l=0$ and $k=n$, then the only partition is
  \begin{displaymath}
    (k,0) = k(1,0),
  \end{displaymath}
  completing the proof.
\end{proof}
\begin{theorem}
\label{th:sign_2row}
  For each $n\geq 2$, the sign representation of $S_n$ occurs in $W_{(k,l)'}(\CC^n)$ if and only if $(k,l)\in \{(n-1,0),(n,0),(n-1,1),(n,1)\}$.
  In all these cases it occurs with multiplicity one.
\end{theorem}
\begin{proof}
  From Corollary~\ref{cor:qstar}, the multiplicity of the sign representation in $W_{(k,l)'}(\CC^n)$ is $q^*_n(k,l)-q^*_n(k+1,l-1)$.
  
  If $l=0$, he second term is $0$. The first term, $q^*_n(k,0)$ is $1$ if $k\in \{n-1,n\}$ and $0$ otherwise, since $(0,0)$ may occur at most once, and the other parts must all be $(1,0)$.

For $l\geq 2$, each vector partition of $(k+1,l-1)$ that contributes to $q^*_n(k+1,l-1)$ must contain at least one $(1,0)$ part. The map replacing one $(1,0)$ by $(0,1)$ is a bijection from partitions of $(k+1,l-1)$ and $(k,l)$. Thus in this case the multiplicity is zero.

When $l=1$, both $q^*_n(k,l)$ and $q^*_n(k+1,l-1)$ are zero unless either $k=n$ or $k=n-1$. When $k=n-1$, we have  $q^*_n(k,l)=2$ and $q^*_n(k+1,l-1)=1$ from Lemma \ref{lem:qval}. When $k=n$, Lemma \ref{lem:qval} again verifies that the multiplicity of the sign representation is $1$.
\end{proof}
By Littlewood's formula (Equation \eqref{al:identity}) we have the following corollary
\begin{corollary}
\label{cor:pleth_twocol}
For a partition $\lambda=(k,l)'$, the coefficient of $s_{\lambda}$ in the expansion of the plethysm $s_{(1^n)}[1+h_1+h_2+\dotsb]$ is
\begin{align*}
r_{\lambda \mu}= 
\begin{cases}
1 & (k,l)\in \{(n-1,0),(n,0),(n-1,1),(n,1)\},\\
0 & \text{otherwise.}
\end{cases}
\end{align*}
\end{corollary}

\subsection{Multiplicity of the sign representation in the restriction of hooks}
\label{subsec:hook}
For all $a,b\geq 0$, let $\lambda = (a|b)$ denote the hook partition $(a+1,1^b)$.
We now consider the multiplicity of the sign representation of $S_n$ in $W_\lambda(\CC^n)$ for $n\geq b+1$. The Pieri identity (see \cite[Theorem 7.15.7]{MR1676282}) expresses the Schur function $s_{(a\mid b)}$ as the alternating sum
\begin{displaymath}
s_{(a\mid b)}= \sum_{j=0}^{b}(-1)^j h_{a+1+j}e_{b-j}.
\end{displaymath}
This relation also holds for the corresponding character polynomials, so that
\begin{align}
\label{eq:hookalt}
S_{(a\mid b)}=\sum_{j=0}^{b}(-1)^j H_{a+1+j}E_{b-j}.
\end{align}
Thus the signed moment of $S_{(a|b)}$ is an alternating sum of signed moments, for which we have the following generating function.
\begin{theorem}
\label{th:gen_hook}
For $k,l\geq 0$, the signed moment of $H_kE_l$ is 
\begin{equation}
  \label{eq:signed-moment-HE}
\sum_{k,l,n}\{H_kE_l\}_nu^kv^lz^n= \frac{\prod_{j\geq 0} (1+u^jz)}{\prod_{j\geq 0} (1-u^jvz)}.
\end{equation}
\end{theorem}
\begin{proof}
From Theorem \ref{th:charHE} and the definition of the signed moment,
\begin{displaymath}
\{H_kE_l\}_n= \sum_{\alpha \vdash n}\frac{1}{z_\alpha} (-1)^{|\alpha|-l(\alpha)}\sum_{\beta \vdash k, \gamma \vdash l} (-1)^{|\gamma|-l(\gamma)}\mch{\alpha}{\beta}{\alpha \choose \gamma}.
\end{displaymath}
This is the coefficient of $u^kv^lz^n$ in $\prod_{i\geq 1}\text{exp}\left(\frac{ -(-z)^i(1-(-v)^i}{1-u^i} \right)$. Recalling that $\text{log} x= \sum_{i\geq 1}\frac{x^i}{i}$ gives us the desired expression. 
\end{proof}
A combinatorial interpretation of the signed moment $\smo{H_{a+1+j}E_{b-j}}_n$ of each term in the alternating sum is given by the following proposition
\begin{prop}
\label{prop:HkEl}
For all $a,b \geq 0$ and $j\leq b$, $\smo{H_{a+1+j}E_{b-j}}_n$ is the number of pairs $(\lambda,\mu)$ such that
 \begin{enumerate}
  \item $\lambda=(\lambda_1,\dotsc,\lambda_{b-j})$, where $\lambda_1\geq \dotsb\geq \lambda_{b-j}\geq 0$,
  \item $\mu=(\mu_1,\dotsc,\mu_{n-b+j})$, with $\mu_1>\dotsb>\mu_{n-b+j}\geq 0$,
  \item $\mid \lambda\mid +\mid\mu\mid=a+1+j$.
  \end{enumerate}
\end{prop}
\begin{proof}
The terms in the denominator of the right hand side of \eqref{eq:signed-moment-HE} contribute to the sequence $\lambda$. There are $b-j$ of these terms (since terms with $v$ in them come only from the denominator and there must be $b-j$ such terms). The terms from the numerator contribute to the sequence $\mu$. Each of these terms contain a distinct power of $u$, and there are $n-b+j$ of them. 
\end{proof}
\begin{figure}
\begin{tikzpicture}
\draw[->,ultra thick] (2,0)--(2,5) node[above]{};
\node [label={[anchor=south east, inner sep=0pt]west:$0$}] at (2, 0)   (0){};
\node [label={[anchor=south east, inner sep=0pt]west:$1$}] at (2, 1)   (1){};
\node [label={[anchor=south east, inner sep=0pt]west:$2$}] at (2, 2)   (2){};
\node [label={[anchor=south east, inner sep=0pt]west:$3$}] at (2, 3)   (3){};
\node [label={[anchor=south east, inner sep=0pt]west:$4$}] at (2, 4)   (4){};
\draw[RoyalBlue, ultra thick] (2,0) -- (3,0);
\draw[red, ultra thick] (1,1) -- (2,1);
\draw[red, ultra thick] (1,3) -- (2,3);
\draw[RoyalBlue, ultra thick] (2,3) -- (3,3);
\draw[RoyalBlue, ultra thick] (2,3.1) -- (3,3.1);
\end{tikzpicture}
\caption{The sequence $\lambda=(3,3,0)$ and $\mu= (3,1)$.}
\label{fig:diag}
\end{figure}
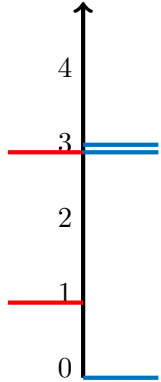

We will think of each sequence $(\lambda,\mu)$ marked on a single axis, with the parts of $\lambda$ marked in {\color{RoyalBlue}blue} and the parts of $\mu$ marked in {\color{red}red}, as in Figure~\ref{fig:diag}. 
We define a sign-reversing involution on the alternating sum in \eqref{eq:hookalt} to obtain a positive combinatorial formula for the multiplicity of the sign representation in $W_{(a\mid b)}(\CC^n)$. 
\begin{theorem}
  \label{theorem:hook-sign}
  For all $a,b\geq 0$, then $\smo{S_{(a|b)}}_n$ is the number of pairs $(\lambda,\mu)$ such that
  \begin{enumerate}
  \item $\lambda=(\lambda_1,\dotsc,\lambda_b)$, where $\lambda_1\geq \dotsb\geq \lambda_b\geq 0$,
  \item $\mu=(\mu_1,\dotsc,\mu_{n-b})$, with $\mu_1>\dotsb>\mu_{n-b}\geq 0$,
  \item $|\lambda|+|\mu|=a+1$,
  \item $\mu_1>\lambda_1$.
  \end{enumerate}
  Equivalently, 
  \begin{displaymath}
    \smo{S_{(a|b)}}_n=\sum_{\rho\in P(a,n)} \binom{r_\rho}{n-b-1},
  \end{displaymath}
  where $P(a,n)$ denotes the set of partitions of $a+n$ with $n$ non-negative parts, and for a partition $\rho\in P(a,n)$, $r_\rho$ is the number of removable cells of $\rho$ (i.e. cells whose removal from the Young diagram of $\rho$ leaves behind a valid partition) that are not in its first row.
\end{theorem}
\begin{proof}
  Let $\cc nab$ denote the set of pairs $(\lambda,\mu)$, where
  \begin{enumerate}
  \item $\lambda=(\lambda_1,\dotsc,\lambda_b)$, where $\lambda_1\geq \dotsb\geq \lambda_b\geq 0$,
  \item $\mu=(\mu_1,\dotsc,\mu_{n-b})$, where $\mu_1>\dotsb>\mu_{n-b}\geq 0$,
  \item $|\lambda|+|\mu|=a$.
  \end{enumerate}
  Then evidently, $|\cc nab|$ is the coefficient of $u^av^bz^n$ in $\frac{\prod_{j \geq 0} (1+u^jz)}{\prod_{j \geq 0} (1-u^jvz)}$, and hence, by Theorem ~\ref{theorem:hook-sign} $\smo{H_aE_b}_n$.
  Note that
  \begin{equation}
    \label{eq:smo-hook}
    \smo{S_{(a|b)}}_n = \sum_{i=0}^b (-1)^i\smo{H_{a+1+i}E_{b-i}}_n = \sum_{i=0}^b (-1)^i |\cc {n}{a+1+i}{b-i}|.
  \end{equation}
  
  Partition $\cc nab$ into two parts: $\gr nab$, which consists of pairs $(\lambda,\mu)\in \cc nab$ such that the largest among the parts of $\lambda$ and $\mu$ occurs in $\lambda$, and $\rr nab$, which consists of pairs $(\lambda,\mu)\in \cc nab$ such that the largest among the parts of $\lambda$ and $\mu$ occurs only in $\mu$.
  \begin{displaymath}
    \cc nab = \gr nab \sqcup \rr nab.
  \end{displaymath}
  Define a bijection
  \begin{displaymath}
    \omega: \gr nab \to \rr n{a+1}{b-1}
  \end{displaymath}
  by setting $\omega(\lambda,\mu)$ to be the pair obtained by moving a largest part of $\lambda$ to a $\mu$ after increasing it by $1$.
  Then $\omega^{-1}(\lambda,\mu)$ is the pair obtained by moving the largest part of $\mu$ to $\lambda$ after reducing it by $1$ (see Figure~\ref{fig:sign_inv}).
\begin{figure}
\begin{tikzpicture}
\draw[->,ultra thick] (2,0)--(2,5) node[above]{};
\node [label={[anchor=south east, draw=red, inner sep=0pt]west:$0$}] at (2, 0)   (0){};
\node [label={[anchor=south east, draw=red, inner sep=0pt]west:$1$}] at (2, 1)   (1){};
\node [label={[anchor=south east, draw=red, inner sep=0pt]west:$2$}] at (2, 2)   (2){};
\node [label={[anchor=south east, draw=red, inner sep=0pt]west:$3$}] at (2, 3)   (3){};
\node [label={[anchor=south east, draw=red, inner sep=0pt]west:$4$}] at (2, 4)   (4){};
\draw[RoyalBlue, ultra thick] (2,0) -- (3,0);
\draw[red, ultra thick] (1,1) -- (2,1);
\draw[red, ultra thick] (1,3) -- (2,3);
\draw[RoyalBlue, ultra thick] (2,3) -- (3,3);
\draw[RoyalBlue, ultra thick] (2,3.1) -- (3,3.1);

\draw[->,ultra thick] (8,0)--(8,5) node[above]{};
\node [label={[anchor=south east, draw=red, inner sep=0pt]west:$0$}] at (8, 0)   (x0){};
\node [label={[anchor=south east, draw=red, inner sep=0pt]west:$1$}] at (8, 1)   (x1){};
\node [label={[anchor=south east, draw=red, inner sep=0pt]west:$2$}] at (8, 2)   (x2){};
\node [label={[anchor=south east, draw=red, inner sep=0pt]west:$3$}] at (8, 3)   (x3){};
\node [label={[anchor=south east, draw=red, inner sep=0pt]west:$4$}] at (8, 4)   (x4){};
\draw[RoyalBlue, ultra thick] (8,0) -- (9,0);
\draw[red, ultra thick] (7,1) -- (8,1);
\draw[red, ultra thick] (7,3) -- (8,3);
\draw[RoyalBlue, ultra thick] (8,3) -- (9,3);
\draw[red, ultra thick] (7,4) -- (8,4);
\end{tikzpicture}
\caption{A demonstration of the involution between the monomials $\color{red}{(u^1z)(u^3z)}\color{RoyalBlue}{(u^0vz)(u^3vz)^2}$ on the left and $\color{red}{(u^1z)(u^3z)(u^4z)}\color{RoyalBlue}{(u^0vz)(u^3vz)}$ on the right}
\label{fig:sign_inv}
\end{figure}
  We may rewrite (\ref{eq:smo-hook}) as
  \begin{align*}
    \smo{S_{(a|b)}}_n & = + |\gr n{a+1}b| + |\rr n{a+1}b|\\
                      & \phantom{=} - |\gr n{a+2}{b-1}| - |\rr n{a+2}{b-1}|\\
                      & \phantom{=} + |\gr n{a+3}{b-1}| - |\rr n{a+3}{b-1}|\\
                      & \phantom{=}\vdots\\
                      & \phantom{=} (-1)^b |\gr n{a+b+1}{0}| + (-1)^b|\rr n{a+b+1}{0}|.
  \end{align*}
  By virtue of the bijection $\omega$, we can cancel out the \textcolor{RoyalBlue}{blue term} in each row of the above expression with the \textcolor{red}{red term} in the row below it.
  Noting further, that $\gr n{a+b+1}{0}=0$, we have
  \begin{displaymath}
    \smo{S_{(a|b)}}_n = |\rr n{a+1}b|,
  \end{displaymath}
  which is precisely the number claimed in the first part of the theorem.

  Given $(\lambda,\mu)\in \rr n{a+1}b$, form a partition $\rho$ by subtracting one from the largest part of the partition $\lambda\cup\mu$ formed by arranging the parts of $\lambda$ and $\mu$ together in weakly decreasing order.
  Also, the parts of $\mu$, excepting the first, form a set of $n-b-1$ distinct parts of $\rho$, including a possible $0$-part.
  Thus they correspond to a choice of $n-b-1$ removable cells of $\rho$ that do not lie in its first row, except when $\rho$ has a part equal to $0$, in which case we can also choose a zero part of $\mu$, giving the expression in the second part of the statement of the theorem.
\end{proof}
\begin{example}
  To compute $\smo{S_{(4,1,1)}}_5$ we take $a=3$, $b=2$, and $n=5$.
  According to Theorem~\ref{theorem:hook-sign}, we must enumerate integers
  \begin{displaymath}
    \lambda_1\geq \lambda_2\geq 0,\quad \mu_1>\mu_2>\mu_3\geq 0,\quad \mu_1>\lambda_1
  \end{displaymath}
  such that
  \begin{displaymath}
    \lambda_1+\lambda_2+\mu_1+\mu_2+\mu_3=4.
  \end{displaymath}
  The only possibilities for $(\lambda,\mu)$ are $((1,0),(2,1,0))$ and $((0,0),(3,1,0))$.
  Therefore $\smo{S_{(4,1,1)}}_5=2$.

  In the second formulation, the sum is over partitions $\rho$ of $3$ with $5$ non-negative parts.
  For $\rho=(3,0,0,0,0)$, $r_\rho=1$; for $\rho=(2,1,0,0,0)$, $r_\rho=2$, and for $\rho=(1,1,1,0,0,0)$, $r_\rho=2$.
  On the other hand $n-b-1=2$.
  Thus we get
  \begin{displaymath}
    \smo{S_{(4,1,1)}}_5 = \binom 12 + \binom 22 + \binom 22 =2.
  \end{displaymath}
\end{example}
\begin{example}
  To compute $\smo{S_{(3,1)}}_3$, take $a=2$, $b=1$, and $n=3$.
  We must count the number of possibilities for
  \begin{displaymath}
    \lambda_1\geq 0,\quad \mu_1>\mu_2\geq 0,\quad \mu_1>\lambda_1,\quad \lambda_1+\mu_1+\mu_2=3.
  \end{displaymath}
  The possibilities for $(\lambda,\mu)$ are $((0),(2,1))$, $((0),(3,0))$, $((1),(2,0))$, so that $\smo{S_{(3,1)}}_3=3$.

  In the second formulation, $P(2,3)={(2,0,0),(1,1,0)}$.
  We have $r_{(2,0,0)}=1$, while $r_{(1,1,0)}=2$, and $n-b-1=1$, so
  \begin{displaymath}
    {S_{(3,1)}}_3 = \binom 11 + \binom 21 =3.
  \end{displaymath}
\end{example}
\begin{corollary}
  For all $a,b\geq 0$,
  \begin{displaymath}
    \smo{S_{(a|b)}}_n > 0 \text{ if and only if } b<n \text{ and } \binom{n-b}2\leq a+1.
  \end{displaymath}
\end{corollary}
\begin{proof}
  Suppose $b<n$, $\binom{n-b}2\leq a+1$, take $\lambda_i=0$ for all $i=1,\dotsc,b$, and take $\mu_1=a+1-\binom{n-b-1}2$, and $\mu_i=n-b-i$ for $1< i\leq n-b$.
  This satisfies the conditions (1) through (4) of Theorem~\ref{theorem:hook-sign}, so that $\smo{S_{(a|b)}}_n > 0$.

  Conversely, if $\lambda_1,\dots,\lambda_b,\mu_1,\dotsc,\mu_{n-b}$ exist satisfying the conditions (1) to (4), then surely $b<n$, and also $\mu_1+\dotsb+\mu_{n-b}\leq a+1$.
  Since $\mu_1>\dotsb>\mu_{n-b}\geq 0$, $\mu_1+\dotsb+\mu_{n-b}\geq \binom{n-b}2$.
  It follows that $\binom{n-b}2\leq a+1$.
\end{proof}

By Littlewood's formula (Equation \eqref{al:identity}) we have the following corollary
\begin{corollary}
\label{cor:pleth_hook_sign}
For all $a,b\geq 0$ with $b+1 \leq n$, the coefficient of $s_{(a|b)}$ in the expansion of the plethysm $s_{(1^n)}[1+h_1+h_2+\dotsb]$ is the number of pairs of partitions $(\lambda,\mu)$ that satisfy the conditions of Theorem \ref{theorem:hook-sign}.
\end{corollary}
\subsection{Multiplicity of the trivial representation in the restriction of hooks}

By the Pieri formula \eqref{eq:hookalt} the moment of $S_{(a|b)}$ is an alternating sum of the moments of the form $H_rE_s$.
The moment of $H_{r}E_{s}$ is a special case of Corollary 4.2 of \cite{ALCO_2021__4_4_703_0}, which we reproduce here.
\begin{theorem}
\label{th:mom_gen_hook}
For integers $k,l\geq 0$, the moment of $H_kE_l$ is 
\begin{displaymath}
\sum_{k,l,n}\mom{H_kE_l}_nu^kv^lz^n= \frac{\prod_{j\geq 0} (1+u^jvz)}{\prod_{j\geq 0} (1-u^jz)}.
\end{displaymath}
\end{theorem}
\begin{proof}
The moment $\mom{H_kE_l}_n$ is the coefficient of $u^kv^lz^n$ in\linebreak $\prod_{i\geq 1}\text{exp}\left(\frac{ z^i(1-(-v)^i}{1-u^i} \right)$. We obtain the generating function from simplifying this exponent as in the proof of Theorem~\ref{th:gen_hook}. 
\end{proof}

%

\begin{theorem}
  \label{theorem:hook-triv}
  For all $a,b\geq 0$, then $\mom{S_{(a|b)}}_n$ is the number of pairs $(\lambda,\mu)$ of partitions such that
  \begin{enumerate}
  \item $\lambda=(\lambda_1,\dotsc,\lambda_{n-b})$, where $\lambda_1\geq \dotsb\geq \lambda_{n-b}\geq 0$,
  \item $\mu=(\mu_1,\dotsc,\mu_{b})$, with $\mu_1>\dotsb>\mu_{b}\geq 0$,
  \item $|\lambda|+|\mu|=a+1$,
  \item $\mu_1<\lambda_1-1$.
  \end{enumerate}
  Equivalently, this $\mom{S_{(a|b)}}_n$ is
  \begin{displaymath}
    \sum_{\rho\in P(a,n)} \binom{r_\rho}{b-1}+\sum_{\rho\in \tilde{P}(a,n)} \binom{r_\rho-1}{b-1},
  \end{displaymath}
  where $P(a,n)$ denotes the set of partitions of $a+n$ with $n$ non-negative parts, $\tilde{P}(a,n)$ denotes the subset of $P(a,n)$ of partitions whose second-largest part is one less than the largest part, and for a partition $\rho\in P(a,n)$, $r_\rho$ is the number of removable cells of $\rho$ that are not in its first row.
\end{theorem}
\begin{proof}
  Let $\cd nab$ denote the set of pairs $(\lambda,\mu)$, where
  \begin{enumerate}
  \item $\lambda=(\lambda_1,\dotsc,\lambda_{n-b})$, where $\lambda_1\geq \dotsb\geq \lambda_{n-b}\geq 0$,
  \item $\mu=(\mu_1,\dotsc,\mu_{b})$, where $\mu_1>\dotsb>\mu_{b}\geq 0$,
  \item $|\lambda|+|\mu|=a$.
  \end{enumerate}
  Then evidently, $|\cd nab|$ is the coefficient of $u^av^bz^n$ in $\frac{\prod_{j \geq 0} (1+u^jvz)}{\prod_{j \geq 0} (1-u^jz)}$, and hence, by Theorem ~\ref{theorem:hook-sign}, equals $\mom{H_aE_b}_n$.
  Note that
  \begin{equation}
    \label{eq:smo-hook}
    \mom{S_{(a|b)}}_n = \sum_{i=0}^b (-1)^i\mom{H_{a+1+i}E_{b-i}}_n = \sum_{i=0}^b (-1)^i |\cd {n}{a+1+i}{b-i}|.
  \end{equation}
The proof proceeds along the lines of that to Theorem \ref{theorem:hook-sign}. Observe that the roles of $\lambda$ and $\mu$ are reversed in this case. Recall that the earlier involution proceeded as in Figure ~\ref{fig:sign_inv} by either increasing a largest blue part (i.e. a part in $\lambda$) by one and changing its colour to red (i.e. appending this part to $\mu$) or reducing a largest red part (i.e. a part in $\mu$) and changing its colour to blue (i.e. appending it to $\lambda$). In this case, the involution attempts to do the reverse, as illustrated in Figure ~\ref{fig:sign_triv}.
\begin{figure}
\begin{tikzpicture}
\draw[->,ultra thick] (2,0)--(2,5) node[above]{};
\node [label={[anchor=south east, draw=red, inner sep=0pt]west:$0$}] at (2, 0)   (0){};
\node [label={[anchor=south east, draw=red, inner sep=0pt]west:$1$}] at (2, 1)   (1){};
\node [label={[anchor=south east, draw=red, inner sep=0pt]west:$2$}] at (2, 2)   (2){};
\node [label={[anchor=south east, draw=red, inner sep=0pt]west:$3$}] at (2, 3)   (3){};
\node [label={[anchor=south east, draw=red, inner sep=0pt]west:$4$}] at (2, 4)   (4){};
\draw[RoyalBlue, ultra thick] (2,0) -- (3,0);
\draw[red, ultra thick] (1,1) -- (2,1);
\draw[red, ultra thick] (1,3) -- (2,3);
\draw[RoyalBlue, ultra thick] (2,3) -- (3,3);
\draw[RoyalBlue, ultra thick] (2,3.1) -- (3,3.1);

\draw[->,ultra thick] (8,0)--(8,5) node[above]{};
\node [label={[anchor=south east, draw=red, inner sep=0pt]west:$0$}] at (8, 0)   (x0){};
\node [label={[anchor=south east, draw=red, inner sep=0pt]west:$1$}] at (8, 1)   (x1){};
\node [label={[anchor=south east, draw=red, inner sep=0pt]west:$2$}] at (8, 2)   (x2){};
\node [label={[anchor=south east, draw=red, inner sep=0pt]west:$3$}] at (8, 3)   (x3){};
\node [label={[anchor=south east, draw=red, inner sep=0pt]west:$4$}] at (8, 4)   (x4){};
\draw[RoyalBlue, ultra thick] (8,0) -- (9,0);
\draw[red, ultra thick] (7,1) -- (8,1);
\draw[RoyalBlue, ultra thick] (8,3) -- (9,3);
\draw[RoyalBlue, ultra thick] (8,3.1) -- (9,3.1);
\draw[RoyalBlue, ultra thick] (8,4) -- (9,4);
\end{tikzpicture}
\caption{A demonstration of the involution between the monomials $\color{red}{(v^1z)(v^3z)}\color{RoyalBlue}{(uv^0z)(uv^3z)^2}$ on the left and $\color{red}{(v^1z)}\color{RoyalBlue}{(uv^0z)(uv^3z)(v^3z)(v^4z)}$ on the right.}
\label{fig:sign_triv}
\end{figure}
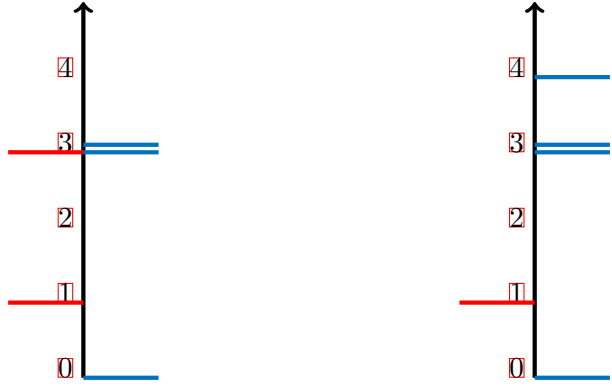

This presents an obstacle however, when the largest red part (i.e. a part in $\mu$) is one less than the largest blue part (i.e. a part in $\lambda$). In this case, the involution attempts to push down a blue part and change its colour to red, but the resulting bipartition is not a valid one since the parts of $\mu$ must be distinct. Thus in this case, we prescribe that the involution ignore the largest part of the bipartition and proceed to the next largest part, as demonstrated in Figure \ref{fig:triv_alt}. 

\begin{figure}
\begin{tikzpicture}
\draw[->,ultra thick] (2,0)--(2,5) node[above]{};
\node [label={[anchor=south east, draw=red, inner sep=0pt]west:$0$}] at (2, 0)   (0){};
\node [label={[anchor=south east, draw=red, inner sep=0pt]west:$1$}] at (2, 1)   (1){};
\node [label={[anchor=south east, draw=red, inner sep=0pt]west:$2$}] at (2, 2)   (2){};
\node [label={[anchor=south east, draw=red, inner sep=0pt]west:$3$}] at (2, 3)   (3){};
\node [label={[anchor=south east, draw=red, inner sep=0pt]west:$4$}] at (2, 4)   (4){};
\draw[RoyalBlue, ultra thick] (2,0) -- (3,0);
\draw[red, ultra thick] (1,1) -- (2,1);
\draw[red, ultra thick] (1,2) -- (2,2);
\draw[RoyalBlue, ultra thick] (2,2) -- (3,2);
\draw[RoyalBlue, ultra thick] (2,3) -- (3,3);
\draw[RoyalBlue, ultra thick] (2,3.1) -- (3,3.1);

\draw[->,ultra thick] (8,0)--(8,5) node[above]{};
\node [label={[anchor=south east, draw=red, inner sep=0pt]west:$0$}] at (8, 0)   (x0){};
\node [label={[anchor=south east, draw=red, inner sep=0pt]west:$1$}] at (8, 1)   (x1){};
\node [label={[anchor=south east, draw=red, inner sep=0pt]west:$2$}] at (8, 2)   (x2){};
\node [label={[anchor=south east, draw=red, inner sep=0pt]west:$3$}] at (8, 3)   (x3){};
\node [label={[anchor=south east, draw=red, inner sep=0pt]west:$4$}] at (8, 4)   (x4){};
\draw[RoyalBlue, ultra thick] (8,0) -- (9,0);
\draw[red, ultra thick] (7,1) -- (8,1);
\draw[RoyalBlue, ultra thick] (8,2) -- (9,2);
\draw[RoyalBlue, ultra thick] (8,3) -- (9,3);
\draw[RoyalBlue, ultra thick] (8,3.1) -- (9,3.1);
\draw[RoyalBlue, ultra thick] (8,3.2) -- (9,3.2);
\end{tikzpicture}
\caption{The involution when the largest blue part cannot be pushed down, since a red part already occupies the level below it.
  Instead, the largest red part is raised up to a blue part.}
\label{fig:triv_alt}
\end{figure}
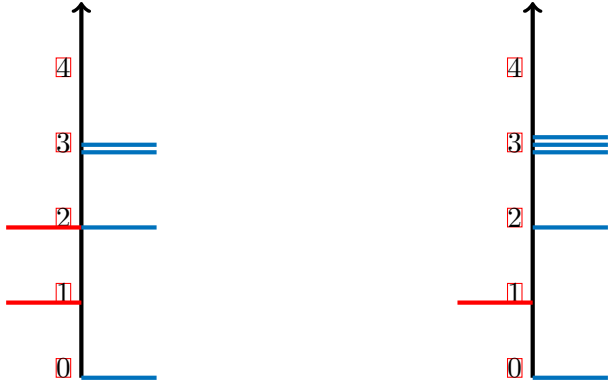
  Partition $\cd nab$ into two parts: $\gd nab$, which consists of pairs $(\lambda,\mu)\in \cd nab$ such that the largest among the parts of $\lambda$ and $\mu$ occurs in $\lambda$ and the largest part in $\mu$ is at least $2$ less than the largest part in $\lambda$ and $\rd nab$, which consists of all other pairs $(\lambda,\mu)\in \cd nab$.
  \begin{displaymath}
    \cd nab = \gd nab \sqcup \rd nab.
  \end{displaymath}
  Then
  \begin{align*}
    \mom{S_{(a|b)}}_n & = + |\gd n{a+1}b| + |\rd n{a+1}b|\\
                      & \phantom{=} - |\gd n{a+2}{b-1}| - |\rd n{a+2}{b-1}|\\
                      & \phantom{=} + |\gd n{a+3}{b-1}| - |\rd n{a+3}{b-1}|\\
                      & \phantom{=}\vdots\\
                      & \phantom{=} (-1)^b |\gd n{a+b+1}{0}| + (-1)^b|\rd n{a+b+1}{0}|.
  \end{align*}
  By virtue of the involution, we can cancel out the \textcolor{RoyalBlue}{blue term} in each row of the above expression with the \textcolor{red}{red term} in the row above it.
  Noting further, that $\rd n{a+b+1}{0}=0$, we have
  \begin{displaymath}
    \mom{S_{(a|b)}}_n = |\gd n{a+1}b|,
  \end{displaymath}
  which is precisely the number claimed in the first part of the theorem.

  The cardinality of $\gd n{a+1}b$ is the same as that of $\rd n{a}{b+1}$. Given $(\lambda,\mu)\in \rd n{a}{b+1}$, form a partition $\rho$ by arranging the parts of $\lambda$ and $\mu$ together in weakly decreasing order. Also, the parts of $\mu$, form a set of $b+1$ distinct parts of $\rho$, including a possible $0$-part.
  Thus they correspond to a choice of $b+1$ removable cells of $\rho$, including either the first largest part of $\rho$, or excluding it but including then the second largest part of $\rho$ (in this case though the largest part is \textcolor{RoyalBlue}{blue}, it cannot be pushed down due to a \textcolor{red}{red} part already occupying the lower position).
\end{proof}
This interpretation allows us to easily recover Theorem 4.8.3 of \cite{ALCO_2021__4_2_189_0}
\begin{corollary}
\label{cor:pos_mom}
  For all $a,b \geq 0$ with $b+1 \leq n$,
  \begin{displaymath}
    \mom{S_{(a|b)}}_n > 0 \text{ if and only if } b<n \text{ and } \binom{b+1}2\leq a.
  \end{displaymath}
\end{corollary}

\begin{corollary}
\label{cor:pleth_hook_triv}
For all $a,b \geq 0$ with $b+1 \leq n$, the coefficient of $s_{(a|b)}$ in the expansion of the plethysm $s_{(n)}[1+h_1+h_2+\dotsb]$ is the number of pairs of partitions $(\lambda,\mu)$ that satisfy the conditions of Theorem \ref{theorem:hook-triv}.
\end{corollary}%
\bibliography{bibfile}
\bibliographystyle{plain}
\end{document}